\documentclass{amsart}
\usepackage{graphicx}
\usepackage{amssymb,amscd,amsthm,amsxtra}
\usepackage{latexsym,esint}
\usepackage{epsfig}

\newtheorem{thm}{Theorem}[section]
\newtheorem{cor}[thm]{Corollary}
\newtheorem{lem}[thm]{Lemma}
\newtheorem{prop}[thm]{Proposition}
\theoremstyle{definition}
\newtheorem{defn}[thm]{Definition}
\theoremstyle{remark}
\newtheorem{rem}[thm]{Remark}
\numberwithin{equation}{section}
\newcommand{\be}{\begin{equation}}
\newcommand{\ee}{\end{equation}}

\newcommand{\R}{\mathbb R}

\newcommand{\eps}{\epsilon}

\newcommand{\p}{\partial}
\newcommand{\di}{\displaystyle}

\newcommand{\comment}[1]{}

\begin{document}

\title[The thin one-phase problem]{Regularity of Lipschitz free boundaries for the thin one-phase problem.}
\author{D. De Silva}
\address{Department of Mathematics, Barnard College, Columbia University, New York, NY 10027}
\email{\tt  desilva@math.columbia.edu}
\author{O. Savin}
\address{Department of Mathematics, Columbia University, New York, NY 10027}\email{\tt  savin@math.columbia.edu}


\begin{abstract}We study regularity properties of the free boundary for the thin one-phase problem which consists of minimizing the energy functional 
$$\label{E} E(u,\Omega) = \int_\Omega |\nabla u|^2 dX + \mathcal{H}^n(\{u>0\} \cap \{x_{n+1} = 0\}), \quad \Omega \subset \R^{n+1},$$ among all functions $u\ge 0$ which are fixed on $\p \Omega$.

We prove that 
the free boundary $F(u)=\p_{\R^n}\{u>0\}$ of a minimizer $u$ has locally finite $\mathcal{H}^{n-1}$ measure and is a $C^{2,\alpha}$ surface except on a small singular set  of Hausdorff dimension $n-3$. 
We also obtain $C^{2,\alpha}$ regularity of Lipschitz free boundaries of viscosity solutions associated to this problem.
 \end{abstract} 
\maketitle

\section{Introduction}

In this paper we study minimizers $u$ of the energy functional $E$ associated to the {\it thin one-phase} problem \be \label{E} E(u,\Omega) := \int_\Omega |\nabla u|^2 dX + \mathcal{H}^n(\{(x,0) \in  \Omega : u(x,0)>0\}),\ee
where $\Omega \subset \R^{n+1} = \R^n \times \R$ and points in $\R^{n+1}$ are denoted by $X=(x,x_{n+1})$. 

We are mainly concerned with the regularity of the free boundary of minimizers $u$, that is the set  $$F(u) : = \p_{\R^{n}}\{u(x,0)>0\} \cap \Omega \subset \R^n. $$

We also consider viscosity solutions to the thin one-phase problem  (see problem \eqref{FBintro} below) and investigate the regularity of Lipschitz free boundaries. 

Throughout  this paper we consider only  domains  $\Omega$ and solutions $u$ such that $$\text{$\Omega$ is
symmetric with respect to $\{x_{n+1} =0\}$},$$$$u \geq 0 \quad \text{is even with respect to $x_{n+1}$}.$$

The thin one-phase problem is closely related to the classical Bernoulli free boundary problem (or one-phase problem) where the second term of the energy $E$ is replaced by $\mathcal H^{n+1}(\{u>0\})$. In our setting the set $\{u=0\}$ occurs on the lower dimensional subspace $\R^n \times \{0\}$ and the free boundary is expected to be $n-1$ dimensional whereas in the classical case the free boundary is $n$-dimensional (lying in $\R^{n+1}$).    
There is a wide literature on the regularity theory for the free boundary in the standard Bernoulli problem which has similarities to the regularity theory of minimal surfaces, see for example \cite{AC, ACF, C1,C2,C3, CJK, CS, DJ1, DJ2}.

The thin one-phase problem was first introduced by Caffarelli, Roquejoffre and Sire in \cite{CRS} as a variational problem involving fractional $H^s$ norms. Such problems are relevant in classical physical models in mediums where long range (non-local) interactions are present, see \cite{CRS} for further motivation. For example, if $u$ is a local minimizer of $E$ defined in $\R^{n+1}$ then its restriction to the $n$-dimensional space $\R^n \times \{0\}$ minimizes locally an energy of the type
$$c_n\|u\|_{H^{1/2}} + \mathcal H^{n}\{ u>0\}.$$

In \cite{CRS} the authors obtained the optimal regularity for minimizers $u$, the free boundary condition along $F(u)$ and proved that, in dimension $n=2$, Lipschitz free boundaries are $C^1$. The question of the regularity of the free boundary in higher dimensions was left open. 
In \cite{DR} De Silva and Roquejoffre studied viscosity solutions of the thin one-phase problem associated to the energy $E$ and showed that flat free boundaries are $C^{1,\alpha}$. Motivated by the present paper, the current authors improved this result to $C^{2,\alpha}$ regularity. This estimate and some basic theorems for viscosity solutions were obtained in \cite{DS} and they play a crucial role in the present paper (see Section 2).  

The thin two-phase problem, that is when $u$ is allowed to change sign, was considered by Allen and Petrosyan in \cite{AP}. They showed that the positive and negative phases  are always separated thus the problem reduces locally back to a one-phase problem. They also obtained a Weiss type monotonicity formula for minimizers and proved that, in dimension $n=2$, the free boundary is $C^1$ in a neighborhood of a regular point. 

The main difficulty in the thin-one phase problem occurs near the free boundary where all derivatives of $u$ blow up and the problem becomes degenerate. 
The method developed by Caffarelli in \cite{C1, C2} for the $C^{1,\alpha}$ regularity of the free boundary in the standard one-phase problem does not seem to apply in this setting. The question of higher regularity is also delicate.
 
In this paper we obtain regularity results for Lipschitz free boundaries based on a Weiss type monotonicity formula and on the $C^{2,\alpha}$ estimates for flat solutions. 
The monotonicity formula is used in a standard blow-up analysis near the free boundary and reduces the regularity question to the problem of classifying {\it global cones} i.e global solutions which are homogenous of degree $1/2$. The $C^{2,\alpha}$ estimate for flat solutions allows us to show that all Lipschitz cones are trivial.  
This general strategy of obtaining regularity of Lipschitz solutions applies also to the classical one-phase problem and to the minimal surface equation, providing different proofs than the ones of Caffarelli \cite{C1} for the one-phase, and of De Giorgi  \cite{DG} for the minimal surface equation.

Our first main result deals with the regularity of the free boundaries for minimizers. We show that $F(u)$ is a $C^{2,\alpha}$ surface except possibly on a small singular set. 

\begin{thm}\label{thm1} Let $u$ be a minimizer for $E$. The free boundary $F(u)$ is locally a $C^{2,\alpha}$surface, except on a singular set  $\Sigma_u \subset F(u)$ of Hausdorff dimension $n-3$, i.e. $$\mathcal{H}^{s}(\Sigma_u) =0 \quad \mbox{ for $s>n-3.$}$$ Moreover, $F(u)$  has locally finite $\mathcal{H}^{n-1}$ measure.\end{thm}

As a corollary we obtain that in dimension $n=2$, free boundaries of minimizers are always $C^{2,\alpha}.$

As mentioned above we also study the regularity of Lipschitz free boundaries of viscosity solutions to the Euler-Lagrange equation associated to the minimization problem for $E$, that is the  following thin one-phase free boundary problem, 

\begin{equation}\label{FBintro}\begin{cases}
\Delta u = 0, \quad \textrm{in $ 
\Omega \setminus \{(x,0) : u(x,0) = 0 \} ,$}\\
\dfrac{\p u}{\p U_0}= 1, \quad \textrm{on  $F(u):= \p_{\R^{n}}\{u(x,0)>0\} \cap \Omega$.} 
\end{cases}\end{equation}Here the free boundary condition reads\begin{equation}\label{nabla_U}
\dfrac{\p u}{\p U_0}(x_0): = \di\lim_{t \rightarrow 0^+} \frac{u(x_0+t\nu(x_0),0)}{\sqrt t},  \quad x_0 \in F(u) \end{equation} with $\nu(x_0)$ the normal to $F(u)$ at $x_0$ pointing toward  $\{x : u(x,0)>0\}.$ 
We prove the following result (see Section 2 for the definition of viscosity solution).

\begin{thm}\label{thm2} Let $u$ be a viscosity solution to \eqref{FBintro} in $B_1$, $0 \in F(u)$ and assume that $F(u)$ is a Lipschitz graph in the $e_n$ direction  with Lipschitz constant $L.$ Then $F(u) \cap B_{1/2}$ is a $C^{2,\alpha}$ graph for any $\alpha<1$ and its $C^{2,\alpha}$ norm is bounded by a constant that depends only on $n,L$ and $\alpha.$\end{thm}

The paper is organized as follows. In Section 2 we introduce notation and recall definitions and some necessary results from \cite{DS} about viscosity solutions to \eqref{FBintro}. Section 3 is devoted to minimizers of $E$. We prove general theorems which were obtained also in \cite{CRS} and \cite{AP}, such as existence, optimal regularity, non-degeneracy, and compactness. 
We also show that minimizers are viscosity solutions to \eqref{FBintro} (with 1 replaced by the appropriate constant).  In Section 4  we prove a Weiss type monotonicity formula for minimizers of $E$ and also for  viscosity solutions to  \eqref{FBintro} which have Lipschitz free boundaries.  Section 5  deals with minimal cones, that is minimizers of $E$ that are homogeneous of degree $1/2$.  We obtain that the only minimal cones in $\R^{2+1}$ are the trivial ones, and from that we deduce our main Theorem \ref{thm1} by a dimension reduction argument. Finally in the last section we use the flatness theorem and the monotonicity formula  and prove Theorem \ref{thm2}.

\section{Viscosity Solutions}

In this section we introduce notation and recall definitions and some necessary results from \cite{DR, DS}. 

\subsection{Notation.} Throughout the paper, constants which depend only on the dimension $n$ will be called universal.  In general, small constants will be denoted by $c$ and large constants by $C,$ and they may change from line to line in the body of the proofs. The dependence on parameters other than $n$ will be explicitly noted.

A point $X \in \R^{n+1}$ will be denoted by $X= (x,x_{n+1}) \in \R^n \times \R$.
A ball in $\R^{n+1}$ with radius $r$ and center $X$ is denoted by $B_r(X)$ and for simplicity $B_r = B_r(0)$.  We use $B_r^+(X)$ to denote the upper ball  $$B_r^+(X):=B_r(X) \cap \{x_{n+1} >0\}.$$ Also, we write $$\mathcal B_r(X) = B_r(X) \cap \{x_{n+1}=0\}.$$

Let $v \in C(\Omega),$ be a non-negative function on a bounded domain $ \Omega \subset \R^{n+1}$. We associate to $v$ the following sets: \begin{align*}
& \Omega^+(v) := \Omega \setminus \{(x,0) : v(x,0) = 0 \} \subset \R^{n+1};\\
& F(v) := \p_{\R^n} \{v(x,0) > 0\} \cap \Omega \subset \R^{n}.
\end{align*}  Often subsets of $\R^n$ are embedded in $\R^{n+1}$, as it will be clear from the context. 

\subsection{Definition and properties of viscosity solutions.} We consider the thin one-phase free boundary problem ($u \geq 0$)

\begin{equation}\label{FB}\begin{cases}
\Delta u = 0, \quad \textrm{in $\Omega^+(u) ,$}\\
\dfrac{\p u}{\p U_0}= 1, \quad \textrm{on $F(u)$}, 
\end{cases}\end{equation}
where 
\be\label{deffb}\dfrac{\p u}{\p U_0}(x_0):=\di\lim_{t \rightarrow 0^+} \frac{u(x_0+t\nu(x_0),0)} {\sqrt t} , \quad \textrm{$X_0=(x_0,0) \in F(u)$}.\ee
Here $\nu(x_0)$ denotes the unit normal to $F(u)$, the free boundary of $u$, at $x_0$ pointing toward $\{u(x,0)>0\}$.

Our notation for the free boundary condition is justified by the following fact. If $F(u)$ is $C^2$ then any function $u \geq 0$ which is harmonic in $\Omega^+(u)$ has an asymptotic expansion at a point $ X_0=(x_0,0)\in F(u),$
$$u(X) = \alpha(x_0) U_0((x-x_0) \cdot \nu(x_0), x_{n+1})  + o(|X-X_0|^{1/2}),$$
where $U_0(t,s)$ is the real part of $\sqrt{z}$. Thus
in the polar coordinates $$t= r\cos\theta, \quad  s=r\sin\theta, \quad r\geq 0, \quad -\pi \leq  \theta \leq \pi,$$ $U_0$ is given by
\begin{equation}\label{U}U_0(t,s) = r^{1/2}\cos \frac \theta 2. \end{equation}
Then, the limit in \eqref{deffb} represents the coefficient $\alpha(x_0)$ in the expansion above 
$$\frac{\p u}{\p U_0}(x_0)=\alpha(x_0)$$ 
and our free boundary condition requires that $\alpha \equiv 1$ on $F(u).$ 

The precise result proved in \cite{DS} (Lemma 7.5) is stated below and will be often used in this paper.

\begin{lem}[Expansion at regular points from one side]\label{oneside}
Let $w \in C^{1/2}(B_1)$ be $1/2$-Holder continuous, $w \ge 0$, with $w$ harmonic in $B_1^+(w)$. 
If $$0 \in F(w), \quad \mathcal B_{1/2}(1/2 e_n) \subset  \{w(x,0)>0\},$$
then 
$$w= \alpha U_0 + o(|X|^{1/2}), \quad \mbox{for some $\alpha>0$}.$$ 
The same conclusion holds for some $\alpha \ge 0$ if $$\mathcal B_{1/2}(-1/2 e_n) \subset \{w=0\}.$$
\end{lem}





We now recall the notion of viscosity solutions to \eqref{FB}, introduced in \cite{DR}. 

\begin{defn}Given $g, v$ continuous, we say that $v$
touches $g$ by below (resp. above) at $X_0$ if $g(X_0)=
v(X_0),$ and
$$g(X) \geq v(X) \quad (\text{resp. $g(X) \leq
v(X)$}) \quad \text{in a neighborhood $O$ of $X_0$.}$$ If
this inequality is strict in $O \setminus \{X_0\}$, we say that
$v$ touches $g$ strictly by below (resp. above).
\end{defn}


\begin{defn}\label{defsub} We say that $v \in C(\Omega)$ is a (strict) comparison subsolution to \eqref{FB} if $v$ is a  non-negative function in $\Omega$ which is even with respect to $x_{n+1}$ and it satisfies
\begin{enumerate} \item $v$ is $C^2$ and $\Delta v \geq 0$ \quad in $\Omega^+(v)$;\\
\item $F(v)$ is $C^2$ and if $x_0 \in F(v)$ we have
$$v (x_0+t\nu(x_0),0) = \alpha(x_0) \sqrt t + o(\sqrt t), \quad \textrm{as $t \rightarrow 0^+,$}$$ with $$\alpha(x_0) > 1,$$ where $\nu(x_0)$ is the unit normal at $x_0$ to $F(v)$ pointing toward $\{v(x,0)>0\}$.\\
\end{enumerate} 
\end{defn}

Similarly one can define a (strict) comparison supersolution. 

\begin{defn}\label{defvisc}We say that $u$ is a viscosity solution to \eqref{FB} if $u$ is a  continuous non-negative function in $\Omega$ which is even with respect to $x_{n+1}$ and it satisfies
\begin{enumerate} \item $\Delta u = 0$ \quad in $\Omega^+(u)$;\\ \item Any (strict) comparison subsolution (resp. supersolution) cannot touch $u$ by below (resp. by above) at a point $X_0 = (x_0,0)\in F(u). $\end{enumerate}\end{defn}

In \cite{DS} we proved optimal regularity for viscosity solutions. Precisely, we have the following lemma.

\begin{lem}[$C^{1/2}$-Optimal regularity]\label{optimal}Assume $u$ solves \eqref{FB} in $B_2$ and $0 \in F(u)$. Then $$u(X) \le C dist(X, F(u))^{1/2} \quad \mbox{$X \in \mathcal B_{1}$.}$$
Moreover, 
$$\|u\|_{C^{1/2}(B_{1})} \le C (1+u(e_{n+1})).$$ 
\end{lem}

The main result in \cite{DS} (see Theorem 1.1 there) is the following flatness theorem, which improves the previous $C^{1,\alpha}$ result obtained in \cite{DR}. 

\begin{thm} \label{mainT}There exists $\bar \eps >0$ small depending only on $n$, such that if $u$ is a viscosity solution to \eqref{FB}  in $B_1$ satisfying
\begin{equation*} \{x  \in \mathcal{B}_1: x_n \leq - \bar \eps\} \subset \{x \in \mathcal{B}_1:  u(x,0)=0\} \subset \{x \in \mathcal{B}_1  : x_n \leq \bar \eps \},\end{equation*} then $F(u) \cap \mathcal B_{1/2}$ is a $C^{2,\alpha}$ graph for every $\alpha \in (0,1)$ with  $C^{2,\alpha}$ norm bounded by a constant depending on $\alpha$ and $n$. \end{thm}

We recall now the definition of a special family of functions $V_{\mathcal S, a, b}$  introduced in \cite{DS} which approximate  solutions quadratically.  

For any $a,b \in \R$ we define the following family of (two-dimensional) functions (given in polar coordinates $(\rho, \beta)$)
\begin{equation}
v_{a,b}(t,s):= (1+\frac{a}{4}\rho+\frac{b}{2}t)\rho^{1/2}\cos \frac{\beta}{2},
\end{equation}
that is 
 $$v_{a,b}(t,s) = (1+\frac{a}{4}\rho+\frac{b}{2}t)U_0(t,s) = U_0(t,s) + o(\rho^{1/2}),$$ with $U_0$ defined in \eqref{U}.

Given a surface $\mathcal S= \{x_n = h(x')\} \subset \R^n$,  we call $\mathcal P_{\mathcal S,X}$ the 2D plane passing through $X=(x,x_{n+1})$ and perpendicular to $\mathcal S$,
that is the plane containing $X$ and generated by the $x_{n+1}$-direction and the normal direction from $(x,0)$ to $\mathcal S$.

We define the family of functions 
\begin{equation}\label{vS}
V_{\mathcal S, a,b} (X): = v_{a,b}(t,x_{n+1}), \quad X=(x,x_{n+1}),
\end{equation}
with $t=\rho\cos \beta, x_{n+1}=\rho\sin \beta$ respectively the first and second coordinate of $X$ in the plane $\mathcal P_{\mathcal S,X}$. In other words, $t$ is the signed distance from $x$ to $\mathcal S$ (positive above $\mathcal S$ in the $x_n$-direction.)

If $$\mathcal S:= \{ x_n = \frac 1 2 (x')^T M x' \},$$ for some $M \in S^{(n-1) \times (n-1)}$  we use the notation $$
 V_{M,a,b}(X):= V_{\mathcal S, a,b} (X). $$

 We define the following class of functions$$\mathcal V_\Lambda^0:= \{V_{M, a,b}  : \  a+b- tr M=0,  \ \ \|M\|, |a|, |b| \leq \Lambda\}.$$ Notice that if we rescale $V = V_{M, a. b}$ that is $$V_\lambda(X)=\lambda^{-1/2}V(\lambda X), $$ then it easily follows  from our definition that $$V_\lambda= V_{\lambda M, \lambda a, \lambda b}.$$

It can also be checked from the definition (see also Proposition 3.3. in \cite{DS}) that  if $V \in \mathcal V_\Lambda^0$ then \be\label{laplaceV}|\Delta V(X)| \leq C \Lambda^2, \quad \text{in $B_{1/2}(e_n)$}.\ee

In the course of the proof of our flatness Theorem \ref{mainT} we also obtained that a solution $u$ can be approximated in a $C^{2,\alpha}$ fashion near $0\in F(u)$ by functions $V \in \mathcal V_\Lambda^0$. The precise statement can be formulated as follows (Theorem 5.2 in \cite{DS}).

\begin{thm}\label{V_approx} Assume $0\in F(u)$ and $F(u)$ is a $C^1$ surface in a neighborhood of $0$ with normal $e_n$ pointing towards the positive side. Then, for any $\alpha \in(0,1)$
\begin{equation*}\label{trap1}V(X-\Lambda r^{2+\alpha} e_n) \le u(X) \le V( X+\Lambda r^{2+\alpha}e_n) \quad \mbox{in $B_r$, \, for all $r$ small,} \end{equation*} for some $V=V_{M,a,b} \in \mathcal V^0_\Lambda$, with $\Lambda$ depending on $u, n$ and $\alpha$. 
\end{thm}

As a consequence of the theorem above we obtain the following Lemma \ref{utau}, which together with the Monotonicity formula (Theorem \ref{MFvisc}) are the main ingredients to prove Theorem \ref{thm2} (see Proposition \ref{trivialcones}). This is the lemma where the $C^{2,\alpha}$ regularity of flat free boundaries is needed. For all the other arguments in this paper the $C^{1,\alpha}$ regularity is sufficient.

\begin{lem} \label{utau}Assume $F(u)$ is $C^1$ in a neighborhood of $X_0=(x_0,0) \in F(u)$ and let $\nu \in \R^n \times \{0\}$ denote the unit normal vector at $x_0$ pointing towards $\{u>0\}.$ Then, for all $\alpha \in (0,1)$, for all $r$ small, and for $K$ depending on $u,\alpha, n$ $$|\p_\tau u(X_0 + r\nu)| \leq K r^{\frac 1 2 + \alpha}$$if $\tau \in \R^n \times \{0\}$ is a tangent unit vector to $F(u)$ at $X_0$, that is $\tau \cdot \nu =0.$
\end{lem}
\begin{proof}Assume for simplicity that $X_0=0$, $\nu=e_n.$ Then, by Theorem \ref{V_approx}, we may assume that 
$$V(X - \Lambda r^{2+\alpha}e_n) \leq u(X) \leq V(X+\Lambda r^{2+\alpha} e_n)$$ with $V= V_{M,a,b} \in \mathcal{V}_\Lambda^0.$ The rescalings $$u_r(X) = r^{-1/2} u(rX), \quad V_r(X) = r^{-1/2} V(rX) = V_{rM,ra,rb}(X) \in \mathcal{V}_{\Lambda r}^0$$ satisfy
$$V_r(X - \Lambda r^{1+\alpha} e_n) \leq u_r(X) \leq V_r(X+ \Lambda r^{1+\alpha} e_n).$$ In $B_{1/2}(e_n)$ we have,
$$|u_r - V_r| \leq \Lambda r^{1+\alpha} \p_n (V_r) \leq C(\Lambda) r^{1+\alpha}$$ and (see \eqref{laplaceV})
$$|\Delta(u_r - V_r)| \leq |\Delta V_r| \leq C(\Lambda) r^2.$$ Thus,
$$|\nabla u_r(e_n) - \nabla V_r(e_n)| \leq C(\Lambda) r^{1+\alpha}.$$ Since, $\nabla V_r(e_n) \in span\{e_n,e_{n+1}\}$ and $\tau \cdot \nabla V_r(e_n) = 0$ if $\tau \in \R^n \times \{0\}, \tau \perp e_n,$  we obtain from the previous inequality that $$|\tau \cdot \nabla u_r(e_n)| \leq K r^{1+\alpha}$$ that is $$|\tau \cdot \nabla u(re_n)| \leq K r^{1/2+\alpha}.$$
\end{proof}

The next remark will be used in the proof of the Monotonicity Formula for viscosity solutions.

\begin{rem}\label{rem0} Using the $C^{1,\alpha}$ estimates in \cite{DR}, we can
approximate $u$ by $U_0$ (instead of $V$) in a $C^{1,\alpha}$ fashion and write
  in the proof above that $$U_0(X - \Lambda' r^{1+\alpha}e_n) \leq u(X) \leq U_0(X+ \Lambda' r^{1+\alpha} e_n).$$ This leads to the conclusion that
$$|\nabla u(X) - \nabla U_0(X) | \leq K' |X|^{\alpha-1/2}, $$
for all $X$ in the two-dimensional plane generated by $e_n$ and $e_{n+1}.$ 
\end{rem}

We conclude this section by recalling the following compactness result (Proposition 7.8 in \cite{DS}.)

\begin{prop}[Compactness]\label{compactness}
Assume $u_k$ solve \eqref{FB} and converge uniformly to $u_*$ in $B_1$, and $\{u_k =0\}$ converges in the Hausdorff distance to $\{u_*=0\}$. Then $u_*$ solves \eqref{FB} as well.
\end{prop}

\section{Preliminaries on Minimizers}

In this section we prove general theorems about minimizers of the energy function $E$, defined by
 \be \label{E1} E(u,\Omega) = \int_\Omega |\nabla u|^2 dX + \mathcal{H}^n(\{u>0\} \cap \{x_{n+1} = 0\}).\ee
Most of the results in this section are contained in \cite{CafRS} and \cite{AP}, such as existence, optimal regularity, non-degeneracy, and compactness. For completeness, we sketch their proofs.  
We also show that minimizers are viscosity solutions to problem \eqref{FBintro} (with 1 replaced by the appropriate constant).  

\begin{defn} We say that $u$ is a (local) minimizer for $E$ in $\Omega \subset \R^{n+1}$, if $u \in H^1_{loc}(\Omega)$ and for any domain $D \subset \subset  \Omega$ and every function $v \in H^1_{loc}(\Omega)$ which coincides with $u$ in a neighborhood of $\Omega \setminus D$ we have
$$E(u,D) \leq E(v,D).$$
\end{defn}

Existence of minimizers with a given boundary data on $\p \Omega$ follows easily from the lower semicontinuity of the energy $E$.

We remark that this minimization problem is invariant under the scaling \be\label{scaling} u_\lambda(X) = \lambda^{-1/2}u(\lambda X),\ee
that is $u$ is a minimizer if and only if $u_\lambda$ is a minimizer.

As already remarked in the introduction, throughout  this paper we consider only  domains  $\Omega$ and minimizers $u$ such that $$\text{$\Omega$ is
symmetric with respect to $\{x_{n+1} =0\}$},$$$$u \geq 0 \quad \text{is even with respect to $x_{n+1}$}.$$

We recall the following notation, which will be used often in this section. We write,
$$\mathcal B_r = B_r \cap \{x_{n+1}=0\},$$
 and for any function $v \geq 0$ we denote by $$\mathcal{B}_r^+(v) := \{v>0\} \cap \mathcal B_r .$$

\begin{lem}\label{lem1}If $u\geq 0$ is a minimizer to $E$ in $B_1$ then $u$ is subharmonic in $B_1$ and harmonic in $B_1^+.$\end{lem}
\begin{proof}
Indeed if $\varphi \geq 0$ is in  $C_0^\infty(B_1)$ then $$\mathcal{H}^n(\mathcal B_1^+(u)) \geq \mathcal{H}^n(\mathcal B_1^+(u-\eps  \varphi)).$$ Thus the minimality of $u$,  $$E(u, B_1) \leq E(u-\eps \varphi, B_1)$$ implies $$\int |\nabla u|^2 dX \leq \int |\nabla(u-\eps \varphi)|^2 dX$$ and hence $$\int \nabla u \nabla \varphi \;dX \leq 0, $$ that is $u$ is subharmonic in $B_1$. Similarly, taking $\varphi \in C_0^\infty(B_1^+)$ we obtain that $u$ is harmonic in $B_1^+$. 
\end{proof}

In view of Lemma \ref{lem1} we can define $u$ pointwise as 

$$u(X) = \lim_{r \rightarrow 0} \fint_{B_r(X)} u \;dY.$$

Optimal regularity and non-degeneracy of a minimizer will follow from the next result.

\begin{lem}\label{lem2}Assume that $u$ minimizes $E$ in $B_2$. If $u(0) \geq C>0$ universal then $B_1 \subset \{u>0\},$ and $u$ is harmonic in $B_1$.\end{lem}

Before the proof we recall the following Sobolev inequality. If $\phi \in H^1(\R^{n+1})$ then 
\be\label{Sob} \int_{\R^{n+1}}|\nabla \phi|^2 dX  \geq c(n) \left( \int_{\R^n \times \{0\}} \phi^{2(1+\delta)} dx\right)^{\frac{1}{1+\delta}}, \quad \delta=\frac{1}{n-1}.\ee

\

\textit{Proof of Lemma $\ref{lem2}.$}
Denote by 
$$a(r) = \mathcal{H}^n(\{u=0\} \cap \mathcal B_r), \quad 1\leq r \leq 2.$$ Let $v$ be the harmonic replacement of $u$ in $B_r$. By minimality,
\be\label{min}\int_{B_r} |\nabla u|^2 dX \leq \int_{B_r} |\nabla v|^2 dX + a(r).\ee We have $$\int_{B_r}|\nabla u|^2 dX = \int_{B_r} (|\nabla v|^2 + 2 \nabla v \cdot \nabla (u-v) + |\nabla(u-v)|^2)dX $$ and hence since $v$ is harmonic and equals $u$ on $\p B_r$
$$\int_{B_r} |\nabla u|^2 dX = \int_{B_r} (|\nabla v|^2 + |\nabla(u-v)|^2)dX .$$ Thus, by the Sobolev inequality \eqref{Sob} and \eqref{min}, the inequality above gives
\begin{align}\label{ar}a(r) \geq \int_{B_r} |\nabla(u-v)|^2 dX & \geq c\left(\int_{\R^n \times \{0\}} (v-u)^{2(1+\delta)} dx\right)^{\frac{1}{1+\delta}}\\\nonumber  & \geq c \left(\int_{\{u=0\} \cap \mathcal B_r}v^{2(1+\delta)}dx\right)^{\frac{1}{1+\delta}}.\end{align}

Since $v\geq 0$ is harmonic in $B_r$ we have $$v(X) \geq c\; v(0) r^{-1}\; dist(X,\p B_r).$$
Thus, in the set $$\{u=0\} \cap \mathcal B_{r-2^{-k}}$$ since $v(0) \geq u(0)$ and $1 \leq r \leq 2$ we have $$v \geq c\; 2^{-k} u(0).$$ Hence from \eqref{ar} we get $$a(r) \geq c \;2^{-2k} u(0)^2 a(r-2^{-k})^{\frac{1}{1+\delta}}.$$ We denote by $$a_k := a(1+2^{-k+1}),$$ thus
\be\label{DeGiorgi}a_{k+1} \leq C 2^{4k} u(0)^{-2(1+\delta)} a_k^{1+\delta}, \quad a_1 \leq C.\ee By De Giorgi iteration, if $u(0) \geq C$ is sufficiently large then $a_k \to 0$ as $k \to \infty.$ Thus $a(1)=0$ and in view of \eqref{min} we get that $u$ is harmonic in $B_1$.
\qed

\

By the scaling \eqref{scaling}, Lemma \ref{lem2} gives that if $u$ is a minimizer in $B_{2r}(X_0),$ with $X_0 \in \{x_{n+1}=0\}$ and $u(X_0) \geq Cr^{1/2}$ then $B_r(X_0) \subset \{u>0\}.$ Thus we immediately obtain the following corollary.

\begin{cor}\label{coruharm}Assume $u$ is a minimizer in $B_2$. Then $u$ is continuous in $B_2$ and thus harmonic in $B_2^+(u)$. Moreover, if  $F(u) \cap \mathcal B_{1} \neq \emptyset,$ then \be\label{d12} u(x,0) \leq C\; dist(x, F(u))^{1/2}, \quad \forall x\in \mathcal{B}_1\ee with $C$ universal.
\end{cor} 

We now easily obtain $C^{1/2}$-optimal regularity of minimizers.

\begin{cor}[Optimal Regularity]\label{optreg}Let $u$ be a minimizer in $B_2$. Then \be\label{C12} \|u\|_{C^{1/2}(B_{1})} \leq C(1+ u( e_{n+1})),\ee with $C$ universal. \end{cor}
\begin{proof} Assume that $F(u) \cap \mathcal B_{1} \neq \emptyset$ otherwise the statement is trivial. We write $u=v+w$ with $v,w$ harmonic in $B_{3/2}^+$ and 
$$v=0 \quad \text{on $\{x_{n+1}=0\}$}, \quad v=u \quad \text{on $\p B_{3/2}^+ \cap \{x_{n+1} >0\}$}$$
$$w=u \quad \text{on $\{x_{n+1}=0\}$}, \quad w=0 \quad \text{on $\p B_{3/2}^+ \cap \{x_{n+1} >0\}.$}$$ Then,
$$\|v\|_{C^{1/2}(B_{1}^+)} \leq Cv(e_{n+1}) \leq C u(e_{n+1}),$$ and by Corollary \ref{coruharm} $$\|w\|_{C^{1/2}(B_{1}^+)} \leq \|u\|_{C^{1/2}(\mathcal B_{3/2})} \leq C.$$
\end{proof}

We now prove non-degeneracy of a minimizer.

\begin{lem}[Non-degeneracy] \label{lem3}Assume $u$ is a minimizer and $$B_1 \subset \{u>0\}.$$ Then, $$u(0) \geq c>0$$ with $c$ universal. \end{lem}
\begin{proof}
Let $\varphi \in C_0^{\infty}(B_{1/2}), \varphi \equiv 1$ in $B_{1/4}.$ Since, $u$ is harmonic in $B_1$
$$\|u\|_{L^\infty(B_{1/2})}, \|\nabla u\|_{L^\infty(B_{1/2})} \leq Cu(0)$$ and we obtain that 
$$\int_{B_1}|\nabla u|^2 dX \geq \int_{B_1} |\nabla (u(1-\varphi))|^2 dX - C u(0)^2.$$ Also, 
$$\mathcal{H}^n(\mathcal B_1^+(u)) \geq \mathcal{H}^n(\mathcal B_1^+(u(1-\varphi)>0)) + c_0.$$ In conclusion, by the minimality of $u$
$$0 \geq -Cu(0)^2 + c_0$$ that is $$u(0) \geq c.$$
\end{proof}

Again by the scaling \eqref{scaling}, the lemma above gives that if $u$ is a minimizer in $B_2$ then $$u(X_0) \geq C dist(X_0,\{u=0\})^{1/2}, \quad \forall X_0 \in \mathcal B_1.$$

In the next lemma, we prove that minimizers satisfy a slightly different type of non-degeneracy which will be used to prove density estimates for the zero phase. 

\begin{lem}\label{lem4} Assume $v \geq 0$ is defined in $B_1,$ harmonic in $B^+_1(v).$ Assume that there is a small constant $\eta>0$ such that \be\label{c12*}\|v\|_{C^{1/2}(B_1)}\leq \eta^{-1},\ee and $v$ satisfies the non-degeneracy condition on $\mathcal B_1$, $$v(X) \geq \eta \; d(X)^{1/2}, \quad X \in \mathcal{B}_1, d(X)= dist(X, \{v=0\}).$$Then if $0 \in F(v)$, $$\max_{\mathcal B_r} v \geq c(\eta)\; r^{1/2}, \quad \forall r \leq 1.$$\end{lem}

\begin{proof} The proof follows the lines of Lemma 7 in \cite{C3} (see also \cite{CRS}.)
Given a point $X_0 \in \mathcal B_1^+(v)$ (to be chosen close to 0) we construct a sequence of points $X_k \in  \mathcal B_1$ such that $$v(X_{k+1})=(1+\delta)v(X_k), \quad |X_{k+1}- X_k| \leq C(\eta) d(X_k),$$ with $\delta$ small depending on $\eta.$

Then using the fact that $d(X_k) \sim v^2(X_k)$ and that $v(X_k)$ grows geometrically we find \begin{align*}|X_{k+1} - X_0| &\leq \sum_{i=0}^{k} |X_{i+1}-X_i|  \leq C \sum_{i=0}^k d(X_{i}) \\ & \leq C \sum_{i=0}^k v^2(X_{i}) \leq C v^2(X_{k+1}) \sim d(X_{k+1}).\end{align*} 
Hence for a sequence of $r_k$'s of size $v^2(X_k)$ we have that $$\sup_{\mathcal B_{r_k}(X_0)} v \geq c r_k^{1/2}$$ from which we obtain that 
 $$\sup_{\mathcal B_{r}(X_0)} v \geq c r^{1/2}, \quad \text{for all $r \geq |X_0|.$}$$ 
The conclusion follows by letting $X_0$ go to 0.

We now show that the sequence of $X_k$'s  exists.  Assume we constructed $X_k$. After scaling we may suppose that $$v(X_k) = 1.$$ We call $Y_k$ the point where the distance from $X_k$ to $\{v=0\}$ is achieved. By the assumptions on $v$ ($C^{1/2}$ bound and non-degeneracy), $$c(\eta) \leq d(X_k)=|X_k-Y_k| \leq C(\eta).$$
Assume by contradiction that we cannot find $X_{k+1}$ in $\mathcal B_M(X_k)$ with $M$ large to be specified later, with $$v(X_{k+1}) \geq 1+\delta.$$ Then $$v \leq 1+\delta +w,$$ with $w$ harmonic in $B_M^+(X_k)$, $$w=0 \quad \text{on $\{x_{n+1} =0\}$}, \quad w=v \quad \text{on $\p B_M(X_k) \cap \{x_{n+1}>0\}$}.$$

We have,
$$w \leq C(n) \frac{x_n}{M} \sup_{B_M^+(X_k)} v \leq C \eta^{-1} x_n M^{-1/2} \leq \delta \quad \text{in $B:=B_{d(X_k)}(X_k),$}$$ if $M$ is chosen large depending on $\delta.$ Thus, \be\label{bound1}v \leq 1+2\delta \quad \text{in $B.$}\ee 
On the other hand, $v(Y_k)=0, Y_k \in \p B$. Thus from the H\"older continuity of $v$ we find \be\label{bound2}v \leq \frac 1 2, \quad \text{in $B_{c(\eta)}(Y_k)$}.\ee

If $\delta$ is sufficiently small \eqref{bound1}-\eqref{bound2} contradict that $$1=v(X_k)= \fint_{B} v.$$ \end{proof}

Next we prove a density estimate for the zero phase of minimizers.

\begin{cor}\label{cordensity}If $u$ is a minimizer in $B_2$ and $0 \in F(u)$ then \be\label{ND}\sup_{\mathcal B_r} u \geq \mu r^{1/2},\ee and $$1-\mu \geq \frac{\mathcal{H}^n(\{u=0\} \cap \mathcal B_r)}{\mathcal{H}^n(\mathcal B_r)} \geq \mu$$ where $\mu$ depends on $n$ and $u( e_{n+1})$.\end{cor}
\begin{proof} By scaling it suffices to prove the corollary only for $r=1$. The first statement is contained in Lemma \ref{lem4}, in view of the optimal regularity and non-degeneracy of minimizers. This easily implies the left inequality in the density estimate.
We now prove the other inequality.

From \eqref{ND}, for some $X_0 \in \mathcal B_{1/8}$, $u(X_0) \geq \frac \mu 2.$ Then, from the proof of Lemma \ref{lem1} with $u(X_0)$ replacing $u(0)$ we see that if $$\mathcal H^n(\{u=0\} \cap \mathcal B_{1/2}(X_0)) \leq \mathcal H^n(\{u=0\} \cap \mathcal B_1) \leq \delta$$ for $\delta$ sufficiently small depending on $\mu$, then by De Giorgi iteration argument (see \eqref{DeGiorgi})$$B_{1/4}(X_0) \subset \mathcal \{u>0\}.$$ This contradicts the fact that $0 \in F(u) \cap \mathcal B_{1/4}(X_0).$
\end{proof}

From the density estimate we immediately obtain the following corollary.

\begin{cor}Let $u$ be a minimizer. Then $\mathcal{H}^n(F(u)) =0.$\end{cor}

\begin{rem}\label{strong} We remark that if $u \in C^{1/2}(B_1) \cap H^1(B_1)$,  $u$ is harmonic in $B^+_1(u)$ and $\mathcal H^n(F(u)) =0$ then $u$ satisfies the following integration by parts identity,
$$\int_{B_1} |\nabla u|^2 dX = \int_{\p B_1} u u_\nu\;d\sigma.$$ To justify this equality we notice that since $u$ is harmonic in $B^+_1$
$$\int_{B_1} |\nabla u|^2 dX = \lim_{\eps \to 0} \int_{B_1 \setminus \{|x_{n+1}| \leq \eps\}}  |\nabla u|^2 dX = \int_{\p B_1} u u_\nu\;d\sigma+\lim_{\eps \to 0} \int_{|x_{n+1}| = \eps} u  u_\nu.$$ However, $$\lim_{\eps \to 0} \int_{|x_{n+1}| = \eps} u  u_\nu=0,$$ since $u |\nabla u| \leq K$ (because $u(X) \leq K dist(X,\{u=0\})^{1/2}$) , $\mathcal H^n(F(u))=0$ and $$\lim_{\eps \to 0} u u_\nu(x,\eps) =0,\quad \text{if $x \not \in F(u)$}.$$
\end{rem}

We now prove a compactness result for minimizers.

\begin{thm}\label{compact} Assume $u_k$ are minimizers to $E$ in $\Omega$ and $u_k \to u$ uniformly locally. Then $u$ is a minimizer to $E$ and $\{u_k=0\} \to \{u=0\}$ locally in the Hausdorff distance.
\end{thm}

\begin{proof} Assume for simplicity $\Omega= B_2.$ Since the $u_k(e_{n+1})$ are uniformly bounded, the $u_k$ are uniformly non-degenerate and $C^{1/2}$ in $B_{1}$ in view of Corollary \ref{optreg}.

First we show that $\{u_k=0\} \to \{u=0\}$ locally in the Hausdorff distance.  If $X_0 \in \mathcal B_1$  and $\mathcal B_\eps(X_0) \subset \{u>0\}$ then by the uniform convergence of the $u_k$, $ \mathcal B_{\eps /2}(X_0) \subset \{u_k>0\}$ for all large $k$. 

If  $\mathcal B_\eps(X_0) \subset \{u=0\}$ then $\mathcal B_{\eps /2}(X_0) \subset \{u_k=0\}$. Otherwise, by Lemma \ref{lem3} $$F(u_k) \cap \mathcal B_{\eps/2}(X_0) \neq \emptyset.$$ Call $Y_k \in F(u_k) \cap \mathcal B_{\eps/2}(X_0)$, then  by the non-degeneracy of the $u_k$, 
$$\sup_{\mathcal B_\eps(X_0)} u_k \geq \sup_{\mathcal B_{\eps/2}(Y_k)} u_k \geq \mu \eps^{1/2}$$ and we reach a contradiction using that $u_k$ converges uniformly to $u.$

In particular 
 $$\chi_{\{u_k>0\}}(x) \to \chi_{\{u>0\}}(x)\quad \text{for all $x \not \in F(u)$}.$$

Next, we show that $$\mathcal H^n(F(u))=0,$$
hence the convergence above holds $\mathcal H^n$-a.e.

Indeed, assume $X_0 \in F(u) \cap \mathcal B_{1}.$ Then we can find $Y_k \in F(u_k)$ such that $Y_k \to X_0.$ From Corollary \ref{cordensity} applied to the $u_k$ on balls centered at the $Y_k$ and the uniform convergence of the $u_k$ we obtain that the limit $u$ satisfies the same estimates in the conclusion of Corollary \ref{cordensity}. 

We now prove that $u$ is a minimizer for $E$. First we notice that $$u_k \to u, \quad \text{in $H^1(B_1)$}.$$
Indeed, since $u_k \to u$ uniformly, we have that $\nabla u_k \rightharpoonup \nabla u$ weakly in $H^1(B_1)$ and by Remark \ref{strong} and Lebesgue dominated convergence theorem,
$$\int_{B_1} |\nabla u_k|^2 \to \int_{B_1} |\nabla u|^2.$$

Let $v \in H^1(B_1)$ with $v =u$ outside $B_{1-\delta},$ and let $\varphi$ be a cut-off function with $\varphi = 1$ in $B_{1-\delta}$ and $\varphi=0$ outside  $B_{1-\delta/2}.$
Define,
$$v_k = \varphi v + (1-\varphi) u_k,$$
then, by the minimality of the $u_k$
$$E(v_k,B_1) \geq E(u_k,B_1).$$ We let $k \to \infty$ in this inequality and use that $$v_k \to v \quad \text{in $H^1$}, \quad \chi_{\{v_k >0\}} \to \chi_{\{v>0\}} \quad \text{$\mathcal H^n$- a.e.}$$
 to obtain the desired inequality $$E(v, B_1) \geq E(u, B_1).$$ 
\end{proof}

Next, we want to prove that minimizers are viscosity solutions. For this purpose we need the following proposition, which we will also use later in our dimension reduction argument in Section 5.

\begin{prop}\label{prop}Assume $u$ is constant in the $e_1$ direction i.e. $$u(x_1,x_2,\ldots, x_{n+1}) = v (x_2,\ldots x_{n+1}).$$ Then, $u$ is a minimizer in $\R^{n+1}$ if and only if $v$ is a minimizer in $\R^n.$\end{prop}
\begin{proof}
Assume $u$ is a minimizer in $\R^{n+1}$ and let $w(x_2,\ldots, x_{n+1})$ be a function which coincides with $v$ outside $B_K \subset \R^n.$ Then define $$\tilde u:= \varphi(x_1)w(x_2,\ldots, x_{n+1}) + (1-\varphi(x_1))v(x_2,\ldots, x_{n+1}),$$ with
$$\varphi(x_1) = \begin{cases}1 \quad \text{if $|x_1| \leq R-1,$}\\ 0 \quad \text{if $|x_1| \geq R.$}\end{cases}$$ Then $\tilde u$ coincides with $u$ outside of $\Omega:= [-R,R] \times B_K.$ Hence,
$$E(u,\Omega) \leq E(\tilde u, \Omega),$$ that implies

$$2R E(v,B_K) \leq 2(R-1) E(w,B_K) + M$$ with $M$ depending on $w$ and $v$ but not on $R$. We let $R \to \infty$ and obtain $$E(v,B_K) \leq E(w,B_K).$$

Viceversa, assume that $v$ is a minimizer in $\R^n.$ Then  if $w = u$ outside of $\Omega$ with $\Omega$ as above, $$E(w, \Omega) \geq \int_{-R}^R E(w(x_1,\cdot), B_K) dx_1.$$ Using that $v$ is a minimizer,
 $$E(w, \Omega) \geq \int_{-R}^R E(v(x_2,\ldots, x_{n+1}), B_K) dx_1= E(u,\Omega).$$\end{proof}

\begin{prop}\label{min-is-visc} If $u$ is a minimizer for $E$  then $u$ is a viscosity solution to $$\begin{cases}\Delta u = 0 \quad \text{in $\{u>0\}$}\\ \ \\ \dfrac{\p u}{\p U_0} = \sqrt{\dfrac 2 \pi} \quad \text{on $F(u)$}.\end{cases}$$
\end{prop}
\begin{proof} The fact that $u$ is harmonic in the set where it is positive is already proved in Corollary \ref{coruharm}. We need to verify the free boundary condition. Assume that we touch $F(u)$ at 0 with $B_\delta(\delta e_n)$ from the positive side (or the zero side.) Then by Lemma \ref{oneside} $u$ has an expansion $$u(X)= \alpha U_0(x_n, x_{n+1}) + o(|X|^{1/2}),$$ with $\alpha >0$ in view of the non-degeneracy \eqref{ND}, (see \eqref{U} for the definition of $U_0$). It suffices to prove that $$\alpha =\sqrt{\frac 2 \pi}.$$
The rescaled solutions $$\lambda^{-1/2} u(\lambda X)$$ converge uniformly to $\alpha U_0$ thus by Theorem \ref{compact} and Proposition \ref{prop}, $\alpha U_0$ is a minimizer in $\R^2.$ The following computations are two-dimensional. We perturb $U_0$ as 
$$V(X) = U_0(X-\eps \varphi(X)e_1), \quad \varphi \in C_0^\infty(B_2), \quad \text{$\varphi \equiv 1 $ in $B_{3/2}$}.$$ Then,
\begin{align*}\int_{B_1} |\nabla V|^2 - \int_{B_1} |\nabla U_0|^2 &=\int_{B_1(-\eps e_1)} |\nabla U_0|^2 - \int_{B_1} |\nabla U_0|^2 \\ &= -\eps \int_{\p B_1} |\nabla U_0|^2 \nu \cdot e_1  + O(\eps^2) \\ & = O(\eps^2)\end{align*}
because $|\nabla U_0|$ is constant on $\p B_1.$
Since $V = U_0 - \eps \varphi(U_0)_1 + O(\eps^2),$ where $(U_0)_\tau$ denotes the derivative of $U_0$ in the $\tau$-direction, we have 
\begin{align*}\int_{B_2 \setminus B_1} |\nabla V|^2 - \int_{B_2 \setminus B_1} |\nabla U_0|^2 &= \int_{B_2 \setminus B_1} 2 \nabla U_0 \cdot \nabla (V-U_0) +2 |\nabla(V-U_0)|^2 dX\\ & = 2\eps \int_{\p B_1} (U_0)_\nu(U_0)_1 + O(\eps^2)\\ &=\frac \eps 2 \int_{\p B_1} (\cos \frac \theta 2 )^2 + O(\eps^2) = \eps \frac{\pi}{2} + O(\eps^2).\end{align*}
In the equality above we used that (see formula \eqref{U}) $$(U_0)_1 =(U_0)_\nu = \frac 1 2 r^{-1/2} \cos\frac \theta 2.$$In conclusion, since $$\mathcal H^1(\{V>0\} \cap B_2) - \mathcal H^1(\{U_0 >0\} \cap B_2) =-\eps$$ we obtain that
$$E(\alpha V,B_2) - E(\alpha U_0, B_2) = \eps(\alpha^2 \frac \pi 2 -1) +O(\eps^2)$$ from which we conclude that $$\alpha^2 \frac \pi 2 -1=0 \quad \text{that is $\alpha =\sqrt{\frac 2 \pi}$},$$ as desired.
\end{proof}

\section{Monotonicity Formula}

In this section we prove a Weiss type monotonicity formula (see \cite{W}) for minimizers of the energy functional $E$ and also for  viscosity solutions to the thin one-phase problem  \eqref{FB} which have Lipschitz free boundaries. 
In the case of minimizers this result is also contained in \cite{AP}.

\begin{thm}[Monotonicity formula for minimizers] \label{MF} If $u$ is a minimizer to $E$ in $B_R,$ then $$\Phi_u(r) := r^{-n}E(u,B_r)- \frac 12 r^{-n-1}\int_{\p B_r} u^2, \quad 0<r \leq R,$$ is increasing in $r.$ Moreover $\Phi_u$ is constant if and only if $u$ is homogeneous of degree $1/2$.\end{thm}

Before the proof, we remark that the rescaling $$u_\lambda(X) := \lambda^{-1/2} u (\lambda X)$$ satisfies \be\label{rescaledmon}\Phi_{u_\lambda}(r) = \Phi_u(\lambda r).\ee
\begin{proof}
For a.e. $r$ we have \begin{align}\label{regular1}
& \frac{d}{dr}\left(\int_{B_r} |\nabla u|^2 dX\right)= \int_{\p B_r} |\nabla u|^2 d\sigma,\\ \label{regular2}
&\frac{d}{dr}\left(\mathcal{H}^n(\{u>0\} \cap \mathcal B_r)\right)= \mathcal{H}^{n-1}(\{u>0\} \cap \p \mathcal B_r),\\ \label{regular3}
&\frac{d}{dr}\left(r^{-n-1} \int_{\p B_r} u^2 d\sigma\right) = r^{-n-2}\int_{\p B_r} (2r uu_{\nu} - u^2) d\sigma,\end{align} where in \eqref{regular3} we used that $u^2$ is a Lipschitz function. This follows from the fact that $u(X) \leq C dist(X,\{u=0\})^{1/2}$ (see Corollary \ref{coruharm}.) 

Assume that the equalities above are satisfied at $r=1$. Define,
$$v_\eps(X) = \begin{cases}(1-\eps)^{1/2} u(\frac{X}{1-\eps}), \quad \text{if $|X| \leq 1-\eps$}, \\ \ \\ |X|^{1/2} u\left(\frac{X}{|X|}\right), \quad \text{if $1-\eps < |X| \leq 1.$}\end{cases}$$ 

We have,
\begin{align*}E(v_\eps, B_1) & =  \int_{B_{1-\eps}} (1-\eps)^{-1} |\nabla u ((1-\eps)^{-1}X)|^2 dX + (1-\eps)^{n}(\mathcal{H}^n(\{u>0\} \cap \mathcal B_1))\\ &+ \eps \int_{\p B_1}\left( \frac 1 4 u^2 + u_\tau^2\right)d\sigma + \eps \mathcal{H}^{n-1}(\{u>0\} \cap \p \mathcal B_1)  + o(\eps),\end{align*} with the sum of the first two terms equaling $(1-\eps)^{n} E(u,B_1).$ In the equality above, $u_\tau$ denotes the tangential gradient of $u$ on $\p B_1.$
Also,
\begin{align*}E(u,B_1) & = \int_{B_{1-\eps}} |\nabla u|^2 dX + \mathcal{H}^n(\{u>0\} \cap \mathcal B_{1-\eps}) \\ & + \eps \left(\int_{\p B_1} |\nabla u|^2 d\sigma +  \mathcal{H}^{n-1}(\{u>0\} \cap \p \mathcal B_1)\right) + o(\eps),\end{align*} with $|\nabla u|^2=u_{\nu}^2 + u_\tau^2.$
The inequality $$E(u,B_1) \leq E(v_\eps, B_1)$$ then implies
$$o(\eps) + \eps \int_{\p B_1}\left( u_\nu^2 - \frac 1 4 u^2 \right)d\sigma + E(u, B_{1-\eps}) \leq (1-\eps)^{n} E(u,B_1).$$ Hence, dividing by $(1-\eps)^n$ and letting $\eps \to 0$ we obtain
$$\frac{d}{dr} (r^{-n} E(u,B_r)) |_{r=1} \geq \int_{\p B_1} \left(u_\nu^2 - \frac 1 4 u^2\right) d\sigma.$$ Using \eqref{regular3}, this shows that  $$\frac{d}{dr} \Phi_u(r)|_{r=1} \geq \int_{\p B_1} (u_\nu - \frac 1 2 u)^2 d\sigma \geq 0.$$ Thus,
$$\frac{d}{dr} \Phi_u(r) \geq 0, \quad \text{a.e. $r$}$$ and the conclusion follows since $\Phi_u$ is absolutely continuous  in $r$. 

From above we see that $\Phi_u$ constant if and only if  $$u_\nu = \frac{1}{2|X|} u, \quad \text{a.e.}$$ which implies that $u$ is homogeneous of degree $1/2.$
\end{proof}

\begin{rem} We used the minimality only up to first order $\eps$ which suggests that the formula remains valid for critical points of $E.$ Indeed, we only need to require that $u$ is  critical for $E$ under domain variations  (see \cite{AP,W}).\end{rem}

Next we show that the Monotonicity formula is valid also for viscosity solutions with Lipschitz free boundary. The proof is technical since we need to justify certain integration by parts.

\begin{thm}[Monotonicity formula for viscosity solutions] \label{MFvisc}Let $u$ be a viscosity solution to $$\begin{cases}\Delta u = 0 \quad \text{in $B_R^+(u)$}\\ \ \\ \dfrac{\p u}{\p U_0} = \sqrt{\dfrac 2 \pi}, \quad \text{on $F(u),$}\end{cases}$$ with $F(u)$ a Lipschitz graph. Then $$\Phi_u(r) := r^{-n}E(u,B_r)- \frac 12 r^{-n-1}\int_{\p B_r} u^2, \quad 0<r \leq R$$ is increasing in $r.$ Moreover $\Phi_u$ is constant if and only if $u$ is homogeneous of degree $1/2$.\end{thm}
\begin{proof} First we remark that  since $\{u=0\}$ is a Caccioppoli set in $\R^n$, $$\mathcal{H}^{n-1}(F(u) \cap \p \mathcal B_r)=0 \quad \text{for a.e. $r$}.$$
We assume that $r=1$ is a regular for value for $\Phi_u$ in the sense of \eqref{regular1}-\eqref{regular3}  and also that the equality above holds i.e. $\mathcal{H}^{n-1}(F(u) \cap \p \mathcal B_1)=0.$ We compute
\begin{align*}\Phi'_u(1) & = \int_{\p B_1} |\nabla u|^2 d\sigma + \mathcal{H}^{n-1}(\{u>0\} \cap \p \mathcal B_1) - n \int_{B_1} |\nabla u|^2 \\ & - n \mathcal{H}^n(\{u>0\} \cap \mathcal B_1) + \int_{\p B_1} - u u_\nu + \frac 1 2 u^2.\end{align*}

Next we want to prove that 
\begin{align}\label{wish} (n-1) \int_{B_1} |\nabla u|^2 dX &= \int_{\p B_1} (|\nabla u|^2 - 2 u_\nu^2)d\sigma \\ \nonumber &-n \mathcal{H}^{n}(\{u>0\} \cap \mathcal B_1) + \mathcal{H}^{n-1}(\{u>0\} \cap \p \mathcal B_1).\end{align}
Using this identity together with the identity (see Remark \ref{strong})
\be\int_{B_1} |\nabla u|^2 dX = \int_{\p B_1} u u_\nu\;d\sigma, \ee
in the formula above for
$\Phi'_u(1)$, we obtain that $$\Phi'_u(1) =2 \int_{\p B_1} (u_\nu - \frac 1 2 u)^2 d\sigma \geq 0.$$ Analogously for a.e. $r$ we get 
 $$\Phi'_u(r) =2 \int_{\p B_r} (u_\nu - \frac 1 2 u)^2 d\sigma \geq 0,$$
from which our conclusion follows.

Let $\Gamma:= F(u).$ To prove \eqref{wish}, we need to show that 
\be\label{firststep} (n-1) \int_{B_1} |\nabla u|^2 dX = \int_{\p B_1} (|\nabla u|^2 - 2 u_\nu^2)d\sigma + \int_{\Gamma \cap B_1} y \cdot \nu_\Gamma d \mathcal{H}^{n-1},\ee 
 with $\nu_\Gamma$ the normal to $\Gamma$ in $\R^n$ pointing toward the positive phase. 
Then, by the divergence theorem,
$$\int _{\Gamma \cap B_1} y \cdot \nu_\Gamma d\mathcal{H}^{n-1} = -n \mathcal{H}^{n}(\{u>0\} \cap \mathcal B_1) + \mathcal{H}^{n-1}(\{u>0\} \cap \p \mathcal B_1).$$ 
This combined with \eqref{firststep} gives us \eqref{wish}.

To prove \eqref{firststep}, let us denote by $$T_\eps := \{X \in \R^{n+1} | dist(X,\Gamma)\leq \eps\}, \quad \Omega_\eps := B_1^+(u) \setminus T_\eps.$$ Notice that $\Omega_\eps$ is a Caccioppoli set and $u$ is a smooth function outside $T_\eps \cup \{u=0\}$. Thus we can use integration by parts. Precisely, 
\be\label{M1} \int_{\Omega_\eps} \nabla u \cdot \nabla(\nabla u \cdot X)\;dX = \int_{\p^* \Omega_\eps} u_\nu \nabla u \cdot X \;d\sigma, \ee where $\p^*\Omega_\eps$ denotes the reduced boundary of $\Omega_\eps$ and $\nu$ denotes the exterior normal to $\p^* \Omega_\eps.$

On the other hand, again using integration by parts we get 

\begin{eqnarray}\label{M2}& \int_{\Omega_\eps} \nabla u \cdot \nabla(\nabla u \cdot X) \;dX = \int_{\Omega_\eps} (u_i u_{ij} x_j + u_i^2)\;dX \\ & \nonumber \  \\ \nonumber &=\int_{\Omega_\eps} \left(-\frac{n+1}{2} |\nabla u|^2 + |\nabla u|^2\right) \; dX + \int_{\p^* \Omega_\eps} \frac 1 2 |\nabla u|^2 X\cdot \nu \; d\sigma.\end{eqnarray}

From \eqref{M1}-\eqref{M2} we find,
\be\label{star}(n-1) \int_{\Omega_\eps} |\nabla u|^2 dX = \int_{\p^* \Omega_\eps} (|\nabla u|^2 X \cdot \nu - 2 u_{\nu} \nabla u \cdot X)\; d\sigma.\ee 

We need to show that \eqref{firststep} follows from the equality above by letting $\eps \to 0.$
We remark that since $u(X) \leq C\; dist(X,F(u))^{1/2}$ (see Lemma \ref{optimal})
$$|\nabla u|^2 \leq C \eps^{-1} \quad \text{on $\p T_\eps$}$$ and since $\Gamma$ is Lipschitz
$$\mathcal{H}^n(\p T_\eps \cap B_r(X_0)) \leq C r^{n-1} \eps, \quad X_0 \in \Gamma.$$
Combining these two inequalities we obtain
\be \label{bound}\left|\int_{\p T_\eps \cap B_r} (|\nabla u|^2 X \cdot \nu - 2 u_\nu \nabla u \cdot X)\; d\sigma\right| \leq C r^{n-1}. \ee Next we claim that  if $\Gamma$ is a $C^{2,\alpha}$ surface in a neighborhood of $X_0 \in \Gamma$ then for $r$ small (depending on the $C^{2,\alpha}$ norm)  we have 
\be\label{limit} \lim_{\eps \to 0} \int_{\p T_\eps \cap B_r(X_0)} (|\nabla u|^2 X \cdot \nu - 2 u_\nu \nabla u \cdot X)\; d\sigma = \int_{\Gamma \cap B_r(X_0)} y \cdot \nu_{\Gamma} d \mathcal{H}^{n-1}\ee with $\nu$ the interior normal direction to $\p T_\eps$ and $\nu_\Gamma$ the normal to $\Gamma$ in $\R^n$ pointing toward the positive phase. To obtain \eqref{limit} we parametrize $T_\eps$ by the map: $$(y, \theta) \to X= y+\eps (\nu_\Gamma \cos \theta + e_{n+1}\sin \theta), \quad (y,\theta) \in \Gamma \times [-\pi, \pi].$$ Then, on $\p T_\eps$ \begin{align*} d \sigma &=  (1+O(\eps)) \;\eps\; dy\;d\theta,\\  X&= y + O(\eps),\\  \nabla u(X) &= \sqrt{\frac 2 \pi}(\nu_\Gamma (U_0)_1 + e_{n+1}(U_0)_2 )+ o(\eps^{-1/2}),\end{align*} where in the last equality (which follows from Remark \ref{rem0}) the derivatives of $U_0$ are evaluated at $\eps \omega$ with $\omega :=(\cos \theta, \sin \theta).$

Using these identities, for a fixed $y \in \Gamma$ we compute, 

\begin{align*}
& \eps  \int_{-\pi}^\pi (|\nabla u|^2 X \cdot \nu - 2 u_\nu \nabla u \cdot X)\; d\theta \\ &= \eps   \int_{-\pi}^\pi (|\nabla u|^2 y \cdot \nu - 2 u_\nu \nabla u \cdot y)\; d\theta +O(\eps)\\
&= \eps \frac 2 \pi   \int_{-\pi}^\pi (|\nabla U_0|^2 \cos \theta y \cdot \nu_\Gamma + 2 (U_0)_\omega(U_0)_1 y \cdot \nu_\Gamma)\; d\theta +O(\eps)\\ &= y \cdot \nu_\Gamma + O(\eps)
\end{align*}
where again the derivatives of $U_0$ are evaluated at $\eps \omega,$ and in the last equality we used that (see the proof of Proposition \ref{min-is-visc})
$$\int_{-\pi}^{\pi} (|\nabla U_0|^2\cos\theta + 2(U_0)_\omega (U_0)_1)\; d\theta = \eps^{-1}\frac \pi 2 .$$
In conclusion, 
$$\eps  \int_{-\pi}^\pi (|\nabla u|^2 X \cdot \nu - 2 u_\nu \nabla u \cdot X)\; d\theta = y \cdot \nu_\Gamma + O(\eps)
$$ and integrating this identity over $\Gamma$ we obtain \eqref{limit}.

From our flatness Theorem \ref{mainT} we know that $\Gamma$ is $C^{2,\alpha}$ except on a closed set $\Sigma$ of $\mathcal H^{n-1}$ measure zero and also recall that $\mathcal H^{n-1}(\Gamma \cap \p \mathcal B_1)=0$. We use a standard covering argument for $\Sigma \cup (\Gamma \cap \p \mathcal B_1)$ with balls of small radius on which we apply the inequality 
\eqref{bound}.  On the remaining part of $\Gamma$ we use  \eqref{limit} and  obtain the desired conclusion $$ (n-1) \int_{B_1} |\nabla u|^2 dX = \int_{\p B_1} (|\nabla u|^2 - 2 u_\nu^2)d\sigma + \int_{\Gamma \cap B_1} y \cdot \nu_\Gamma d \mathcal{H}^{n-1},$$ by passing to the limit as $\eps \to 0$ in \eqref{star}. 
\end{proof}

\begin{rem}\label{rem1} If $u_k$ are minimizers which converges uniformly to $u$ on compact sets, then it follows from the proof of the compactness Theorem \ref{compact} that $$\Phi_{u_k} (r) \to \Phi_u(r).$$ The result is true also if the $u_k$ are viscosity solutions with Lipschitz free boundaries with uniform Lipschitz bound.
\end{rem}

\begin{rem}\label{rem2} If $u$ satisfies either the assumptions of Theorem \ref{MF} or Theorem \ref{MFvisc}  then $\Phi_u(r)$ is bounded below as $r \to 0$. Indeed, by scaling we only need to check that $\Phi_u(1)$ is bounded which follows from the formula below (see Remark \ref{strong})
$$\Phi_u(1) = \int_{\p B_1}(u u_\nu - \frac 1 2 u^2) d\sigma + \mathcal{H}^n(\{u>0\} \cap \mathcal B_1).$$ This means that $$\Phi_u(0^+) = \lim_{r \to 0^+} \Phi_u(r)= \lim_{r \to 0^+} r^{-n} \mathcal H^n(\{u>0\} \cap  \mathcal B_r) \quad \text{exists}$$ and any blow-up sequence $u_\lambda$ converges uniformly on compact sets (up to a subsequence) to a homogeneous of degree $1/2$ solution $U$ (see \eqref{rescaledmon}). 
\end{rem}

\begin{defn}\label{def-cones} A minimizer $U$ of $E$ which is homogeneous of degree 1/2 is called a {\it minimal cone}. Analogously a viscosity solution to \eqref{FB} which is homogeneous of degree 1/2 and has Lipschitz free boundary is called a {\it Lipschitz viscosity cone.} \end{defn}

Let $U$ be a  (minimal or viscosity)  cone. We denote by $\Phi_{U}$ its energy (which is a constant for all $r$) \be\label{omegan}\Phi_{U} = \mathcal{H}^n(\{U>0\} \cap \mathcal B_1) \in (0,\omega_n),\ee
 where $\omega_n$ denotes the volume of the $n$-dimensional unit ball.
 
We say that a cone $U$ is trivial, if it coincides (up to a rotation) with the cone $U_0(X) = U_0(x_n,x_{n+1})$ (defined in \eqref{U}), and therefore its free boundary is a hyperplane. The energy of the trivial cone is $\omega_n/2$.

\section{Minimal Cones}

This section is devoted to the study of minimal cones. First we prove an ``energy gap" result in the spirit of the analogue for minimal surfaces. We then show that in dimension $n=2$ the only minimal cone is the trivial cone $U_0$ (see \eqref{U}). Finally, by a standard dimension reduction argument we prove our main Theorem \ref{thm1}.

\begin{lem}\label{uniformcones} Minimal cones are uniformly $C^{1/2}.$
\end{lem}
\begin{proof}Let $U$ be a minimal cone. From the proof of the $C^{1/2}$ bound (see Corollary \ref{optreg}) we obtain 
$$\frac{|U(X) - U(Y)|}{|X-Y|^{1/2}} \leq C(1+ U(e_{n+1}) |X-Y|^{1/2}), \quad X,Y \in B_{1},$$ with $C$ universal. Writing this estimate for the rescaling $U_R$ $$U_R(\tilde X) = R^{-1/2} U(R \tilde X), \quad X=R\tilde X, \tilde X \in B_2,$$ we obtain
$$\frac{|U(X) - U(Y)|}{|X-Y|^{1/2}} \leq C(1+ \frac 1 R U(R  e_{n+1}) |X-Y|^{1/2}).$$ Since $U$ is homogeneous of degree $1/2$, $$\frac 1 R U(R  e_{n+1}) \to 0, \quad \text{as $R \to \infty,$}$$ and we obtain the desired bound.
\end{proof}

\begin{defn}
Given a minimizer $u$ for $E$ in $\Omega \subset \R^{n+1}$, we say that a point $X \in F(u)$ is a {\it regular point} if there exists a blow-up sequence of $u$ centered at $X$ which converges to the trivial cone. The points of $F(u)$ which are not regular points, will be called {\it singular point } and the set of all singular points of $F(u)$ is denoted by $\Sigma_u.$ \end{defn}

We notice that in view of our flatness Theorem \ref{mainT}, $F(u)$ is a $C^{2,\alpha}$ surface in a neighborhood of any regular point, and moreover $\Sigma_u$ is a closed set in $\Omega.$ 

\begin{prop}[Energy Gap]\label{EE} Let $U$ be a non-trivial minimal cone. Then, there exists a $\delta>0$ universal such that $$\Phi_U \geq \frac{\omega_n}{2} + \delta.$$\end{prop}
\begin{proof} 
First we show that 
$$\Phi_U > \frac{\omega_n}{2}.$$ Assume by contradiction that this does not hold and let $X_0 \in F(U)$ be a point where we can touch $F(U)$ with a ball completely contained in $\{U>0\}$. Call $$\Phi_U(r,X_0) = \Phi_{\bar U}(r), \quad \bar U(X) = U(X -X_0).$$

Then, by \eqref{rescaledmon} and the fact that $U$ is a cone we obtain that $$\Phi_U(r, X_0) = \Phi_{U_r}(1, \frac{X_0}{r}) = \Phi_U(1, \frac{X_0}{r}).$$ Thus,
$$\lim_{r\to \infty} \Phi_U(r, X_0) = \Phi_U \leq \frac{\omega_n}{2}.$$ 
On the other hand, from the expansion of $U$ near $X_0$ (see Theorem \ref{oneside}) the blow-up energy $$\lim_{r \to 0} \Phi_U(r, X_0) = \frac{\omega_n}{2}.$$ By the monotonicity of $\Phi_U(r,X_0)$ we obtain that  $$\Phi_U(r,X_0) \equiv  \frac{\omega_n}{2},$$
and hence $U$ is a cone with respect to $X_0$, thus $U$ is the trivial cone, a contradiction.

Now we prove the existence of $\delta$ by compactness. If no such $\delta$ exists then we can find a sequence of cones $U_k$ with $\Phi_{U_k} \to \omega_n/2.$ By Lemma \ref{uniformcones} we may assume that $U_k \to U_*$ uniformly on compact sets. Thus $\Phi_{U_*} = \omega_n/2$ and hence $U_*$ is the trivial cone in view of the preceding argument. By the flatness Theorem \ref{mainT} and the compactness Theorem \ref{compact}, $F(U_k)$ are smooth in $B_1$ for all large $k$, a contradiction.
\end{proof}

\begin{lem}\label{dimred}Assume $U$ is a minimal cone in $\R^{n+1}$ and $X_0=e_1 \in F(U)$. Then, any blow-up sequence $$V_\lambda(X) = \lambda^{-1/2} U(X_0+\lambda X)$$ has a subsequence $V_{\lambda_k}, \lambda_k \to 0$ which converges uniformly on compact sets to $v(x_2,\ldots,x_{n+1})$ with $V$ a minimal cone in  $\R^{(n-1)+1}$. Moreover if $X_0$ is a singular point for $F(U)$, then $V$ is a non-trivial cone.\end{lem}

\begin{proof} In view of Remark \ref{rem2} and Proposition \ref{prop}, we only need to show that $V$ is constant in the $e_1$ direction. 

From the fact that $U$ is homogeneous of degree 1/2 and from the formula for $V_\lambda$ we get that
\begin{align*}V_\lambda(X) &= \lambda^{-1/2} (1+t\lambda)^{-1/2} U((1+ t\lambda)(X_0+\lambda X))\\ &= (1+t\lambda)^{-1/2}V_\lambda(tX_0 +(1+t\lambda)X).\end{align*}

Letting $\lambda=\lambda_k \to 0$ we obtain that 
$$V(X) = V(tX_0 +X), \quad \text{for all $t$}.$$
Thus, $V$ is constant in the $X_0=e_1$ direction.

The final statement follows from the flatness Theorem \ref{mainT}.
\end{proof}

Assume that  $U$ is a non-trivial minimal cone in $\R^{n+1}$ for some dimension $n$. Then by Lemma \ref{dimred} we obtain that if  $F(U)$ has a singular point different than the origin, then there exists a  non-trivial minimal cone in $\R^{(n-1)+1}.$ By repeating this dimension reduction argument, we can assume that there is a dimension $k\leq n$ and a non-trivial cone in $\R^{k+1}$ which is regular at all points except at 0.

Clearly, all minimal cones in dimension $n=1$ are trivial. 
In the next theorem we show that there are no non-trivial minimal cones in $\R^{2+1}.$

\begin{thm}\label{dimension2}If $n=2$, all minimal cones are trivial. \end{thm}
\begin{proof}
We follow the strategy in \cite{SV}, where the authors proved that non-local minimal cones (defined in \cite{CRSa}) are trivial in $\R^2.$

Let $U$ be a minimal cone. By the discussion above $\Sigma_U=0$. Define, 
$$\psi_R(t) := \begin{cases}1 \quad 0 \leq t \leq R, \\ \ \\2-\dfrac{\log t}{\log R}\quad R < t \leq R^2, \\ \ \\ 0 \quad t \geq R^2.\end{cases}$$
The function $\psi_R$ is a Lipschitz continuos function with compact support in $\R.$

Notice that 

$$\psi'_R(t) = \begin{cases}0 \quad t \in (0,R) \cup (R^2,\infty), \\ \ \\  \dfrac{-1}{t\log R} \quad t\in(R,R^2).\end{cases}$$

We define a bi-Lipschitz change of coordinates:

$$Y:= X+ \psi_R (|X|) e_1$$
and let $$U_R^+(Y) = U(X).$$ Next we estimate $E(U^+_R,B_{R^2})$ in terms of $E(U,B_{R^2})$. We have,
$$D_X Y = I+A$$ with $$A(X) = \psi_R'(|X|)\begin{pmatrix}
  \frac{x_1}{|X|} & \frac{x_2}{|X|} & \cdots & \frac{x_{n+1}}{|X|} \\
  0 & 0 & \cdots & 0 \\
  \vdots  & \vdots  & \ddots & \vdots  \\
  0 & 0 & \cdots & 0
 \end{pmatrix} $$ and $$\|A\| \leq |\psi'_R(X)| << 1.$$
Notice that 
$$D_Y X = (I+A)^{-1} = I - \frac{1}{1+trA} A.$$
We have,
$$\nabla_Y U^+_R = \nabla_X U \;D_YX, \quad dY = (1+tr A) dX,$$
thus
$$|\nabla U_R^+|^2 dY = \nabla U\left(I(1+trA) - (A+A^T) + \frac{1}{1+trA} AA^T \right) (\nabla U)^T dX,$$
and
$$\mathcal{H}^n(\{U^+_R > 0\} \cap \mathcal B_{R^2}) = \int_{\{U>0\} \cap \mathcal B_{R^2}} (1+trA) dx.$$
Writing the same equalities for $U_R^-$ which is defined as $U_R^+$ but changing $\psi_R$ into $-\psi_R$ thus $A$ into $-A$ we obtain,
$$E(U_R^+,B_{R^2}) + E(U_R^-, B_{R^2}) \leq 2 E(U,B_{R^2}) + C \int_{B_{R^2}} |\nabla U|^2 \|A\|^2 dX$$
with
\begin{align*} \int_{B_{R^2}} |\nabla U|^2 \|A\|^2 dX & = \int_R^{R^2} \left(\int_{\p B_r}  |\nabla U|^2 \|A\|^2 d\sigma \right) dr \\ & \leq \int_R^{R^2} C r^2 r^{-1} (\frac{r^{-1}}{\log R})^2 dr \leq \frac{C}{\log R} \to 0, \quad \text{as $R \to \infty.$}\end{align*} 

The inequality above is the crucial step where we used that $n=2$. In conclusion, since $E(U_R^\pm, B_{R^2}) \geq E(U, B_{R^2})$ we get
$$E(U_R^+,B_{R^2}) \leq E(U, B_{R^2}) + \delta(R)$$
with $\delta(R) \to 0$ as $R \to \infty.$
Now the proof continues as in \cite{SV}. We sketch it for completeness. Since $$E(\underline w, B_{R^2}) + E(\bar w, B_{R^2}) = E(U, B_{R^2}) + E(U_R^+, B_{R^2}),$$
with 
$$\underline w :=\min\{U,U_R^+\}, \quad \bar w = \max\{U, U_R^+\},$$
the inequality above shows that \be\label{crucial}E(\underline w,B_{R^2}) \leq E(U, B_{R^2}) + \delta(R).\ee 

We remark that $\{U=0\}$ consists of a finite number of closed sectors, since $\Sigma_U=0.$

Now, assume by contradiction that $U$ is non-trivial. Then we can find a direction (say $e_1$) and either a point $P \in \{U=0\}^o$ such that $P \pm e_1 \in \{U>0\}$ or a point  $P \in \{U>0\}$ such that $P \pm e_1 \in \{U=0\}^o$. Assume for simplicity that we are in the first case. This implies that  \begin{align*}\underline w&=U < U_R^+ \quad \text{ in  neighborhood of $P,$}\\ \underline w&=U_R^+ < U \quad \text{ in  neighborhood of $P-e_1.$}\end{align*}

In conclusion, $\underline w$ is not harmonic in $B^+_{|P|+2}$ and therefore we can modify $\underline w$ inside this ball without changing its values on $\{x_{n+1}=0\}$ so that the resulting function $v$ satisfies
$$E(v, B_{|P|+2}) \leq E(\underline w, B_{|P|+2}) -\eta$$
with $\eta$ small independent of $R.$

In conclusion, using \eqref{crucial} we obtain
$$E(v,B_{R^2}) \leq E(U, B_{R^2}) + \delta(R) -\eta,$$
and we contradict the minimality of $U$ for $R$ large enough.
\end{proof}

By our flatness Theorem \ref{mainT}, Remark \ref{rem2} and the compactness Theorem \ref{compact}, we immediately obtain the following Corollary.

\begin{cor} Minimizers to $E$ in $\R^{2+1}$ have $C^{2,\alpha}$ free boundaries.\end{cor}

In the next two lemmas, we follow the dimension reduction argument due to Federer for minimal surfaces  (see also \cite{CRSa}), and prove the first claim in  Theorem \ref{thm1}, that is 
$$\mathcal H^{s}(\Sigma_u) =0, \quad s>n-3$$ for all minimizers $u$ of $E$ in $\Omega \subset \R^{n+1}.$

\begin{lem}\label{first-cones} Assume that for some $s>0$, $\mathcal H^s( \Sigma_U) =0$ for all minimal cones $U$ in $\R^{n+1}$. Then $\mathcal H^s(\Sigma_v)=0$ for all minimizers $u$ of $E$ defined on  $\Omega \subset \R^{n+1}.$ 
\end{lem}
\begin{proof} First we show the following property $(P)$:
for every $Y \in \Sigma_u$ there exists $d_Y>0$ such that for any $\delta \leq d_Y$, any subset $D$ of $\Sigma_u \cap \mathcal B_\delta(Y)$ can be covered by a finite number of balls $B_{r_i}(Y_i)$, $Y_i \in D$ such that $$\sum_i r_i ^s \leq \frac{\delta^s}{2}.$$

Property $(P)$ follows by compactness. Indeed, given $Y \in \Sigma_u$, assume that the conclusion does not hold for a sequence of $\delta_k \to 0$. By possibly passing to a subsequence, we may assume that the  sequence 
$u_{\delta_k}$ converges uniformly to a minimal cone $U$ where $$u_{\lambda}(X) = \lambda^{-1/2} u(Y+\lambda X).$$ By our hypothesis, we can cover $\Sigma_U \cap \mathcal B_1$ by  a finite number of balls $\mathcal B_{r_i/4}(X_i)$ with radius $r_i/4$ so that $$\sum_i r_i^s \leq \frac 12.$$  
On the other hand, by the flatness Theorem \ref{mainT}, $$\Sigma_{u_{\delta_k}} \cap \mathcal B_1 \subset \bigcup_i \mathcal B_{r_i/2}(X_i)$$ for all large $k$. Thus, after scaling, $u$ satisfies the conclusion in $\mathcal B_{\delta_k}$ for all large $k$ and we reach a contradiction. 

Next, denote by $D_k$ the set of $Y \in \Sigma_v$ with $d_Y \geq 1/k.$ Fix $Y_0 \in D_k.$ By property $(P)$, we can cover $D_k \cap \mathcal B_{r_0}(Y_0)$, $r_0=1/k$ with a finite number of balls $\mathcal B_{r_i}(Y_i)$ with $Y_i \in D_k$ and $$\sum_i r^{s}_i \leq \frac 1 2 r_0^s.$$

Now, we repeat the same argument for each ball $\mathcal B_{r_i}(Y_i)$ and cover it with balls $\mathcal B_{r_{ij}}(Y_{ij})$ with $Y_{ij} \in D_k$ and $$\sum_j r_{ij}^s \leq \frac 1 2 r_i^s.$$ By repeating this argument $m$ times we obtain that
$$\mathcal H^s(D_k \cap \mathcal B_{r_0}(Y_0)) =0,$$ hence $\mathcal H^s(D_k)=0$
and the conclusion follows by letting $k \to \infty$.
\end{proof}

\begin{lem} Assume that for some $s>0$, $\mathcal H^s( \Sigma_U) =0$ for all minimal cones $U$ in $\R^{n+1}$. Then $\mathcal H^{s+1}(\Sigma_V) =0$ for all minimal cones $V$ defined in $\R^{n+2}.$
\end{lem}
\begin{proof}It suffices to show that $\mathcal H^s(\Sigma_V \cap \p \mathcal B_1) = 0$. Using our assumption we can deduce by the same compactness argument in the previous lemma, that  when restricted to $\p \mathcal B_1$, $\Sigma_V\cap \p \mathcal B_1$ satisfies the same property $(P)$ as above. The conclusion now follows again with the same argument as in Lemma \ref{first-cones}.\end{proof}

In dimension $n=3$, in view of Theorem \ref{dimension2}, $\mathcal H^s(\Sigma_U)=0$ for all $s>0$, for all minimal cones $U$. This fact, combined with the previous two lemmas gives the desired claim that 
$$\mathcal H^{s}(\Sigma_u) =0, \quad s>n-3$$ for all minimizers $u$ in $\R^{n+1}.$

Next we show the second claim in Theorem \ref{thm1}, that is $F(u)$ has locally finite $\mathcal H^{n-1}$ measure for al minimizers $u$ in $\R^{n+1}$.

\begin{lem} \label{last-cone}Assume $u$ is a minimizer in $B_2,$ with $\|u\|_{C^{1/2}} \leq M.$ Then, there exists $C(M)$ large depending on $M$ such that $$\mathcal H^{n-1}((F(u) \cap \mathcal B_1) \setminus \cup_ {i=1}^m \mathcal B_{\delta_i}(X_i)) \leq C(M),$$ for some  finite collection of balls $\mathcal B_{\delta_i}(X_i)$ with $$\sum_{i=1}^m \delta_i^{n-1} \leq 1/2.$$
\end{lem}
\begin{proof} Assume by contradiction that we can find $u_k$ such that  $\|u_k\|_{C^{1/2}} \leq M$  and \be\label{hn-1}\mathcal H^{n-1}((F(u_k) \cap \mathcal B_1) \setminus \cup_ {i=1}^m \mathcal B_{\delta_i}(X_i)) \geq k,\ee for any collection of balls with $$\sum_{i=1}^m \delta_i^{n-1} \leq 1/2.$$
We may assume that $u_k$ converges uniformly on compact subsets of $B_2$ to a minimizer $u$. Since $\mathcal H^{n-1}(\Sigma_u) =0$ and $\Sigma_u$ is closed, $$\Sigma_u \cap \mathcal B_1 \subset \cup_{i=1}^m \mathcal B_{\delta_i/2}(X_i), \quad \sum_{i=1}^m\delta_i^{n-1} \leq 1/2,$$ for some collection of balls.

Since $F(u) \setminus \Sigma_u$ is locally a $C^{2,\alpha}$ surface, we conclude from the flatness Theorem that $(F(u_k)\cap \mathcal B_1) \setminus \cup_ {i=1}^m \mathcal B_{\delta_i}(X_i)$
is a $C^{2,\alpha}$ surface which converges in the $C^2$ norm to  $(F(u)\cap \mathcal B_1) \setminus \cup_ {i=1}^m \mathcal B_{\delta_i}(X_i)$ and we contradict \eqref{hn-1}.
\end{proof}

\begin{lem}\label{l_f}Assume $u$ is a minimizer in $B_2$, with $\|u\|_{C^{1/2}} \leq M.$ Then $$\mathcal H^{n-1}(F(u) \cap \mathcal B_1) \leq 2 C(M).$$ \end{lem}
\begin{proof}
By Lemma \ref{last-cone}, $$F(u) \cap \mathcal B_1 \subset \Gamma \cup \bigcup_{i=1}^m \mathcal B_{\delta_i}(X_i)$$ with $\mathcal H^{n-1}(\Gamma) \leq C(M),$ and  $$\sum_{i=1}^m \delta_i^{n-1} \leq 1/2.$$ For each ball $\mathcal B_{\delta_i}(X_i)$ we apply again Lemma \ref{last-cone} rescaled and obtain that
$$F(u) \cap \mathcal B_{\delta_i}(X_i) \subset \Gamma_i \cup \bigcup_{j=1}^{m_i} \mathcal B_{\delta_{ij}}(X_{ij})$$ with 
 $\mathcal H^{n-1}(\Gamma_i) \leq C(M)\delta_i^{n-1},$ and  $$\sum_{j=1}^{m_i} \delta_{ij}^{n-1} \leq \frac 1 2 \delta_i^{n-1}.$$ Now for each ball $\mathcal B_{\delta_{ij}}(X_{ij})$ we apply the same argument and after $l$ such steps we find that 
 $$F(u) \cap \mathcal B_1 \subset \tilde \Gamma \cup \bigcup_{q=1}^r \mathcal B_{\delta_q}(X_q)$$ with $$\mathcal H^{n-1}(\tilde \Gamma) \leq C(M)(1 + \frac 1 2 + \ldots + \frac{1}{2^{l-1}})$$ and $$\sum_q \delta_q^{n-1} \leq 2^{-l},$$ which implies the conclusion.
\end{proof}

\begin{rem}The same argument as above can be used to show that the non-local minimal surfaces defined in \cite{CRSa} have locally finite $\mathcal H^{n-1}$ measure.\end{rem}

{\it Proof of Theorem \ref{thm1}.} The proof follows from Theorem \ref{dimension2} and Lemmas \ref{first-cones}-\ref{l_f}.
\qed

\section{Viscosity Solutions with Lipschitz Free Boundaries}

In this Section we prove our main Theorem \ref{thm2}, that is Lipschitz thin free boundaries are $C^{2,\alpha}.$ First we prove non-degeneracy of viscosity solutions with Lipschitz free boundaries.

\begin{lem} Assume $u$ is a viscosity solution in $B_2$ with $F(u)$ a Lipschitz graph in the $e_n$ direction  with Lipschitz constant L, $0 \in F(u)$.  Then, 
$$\|u\|_{C^{1/2}(B_1)} \leq  C(L)$$
and
$$\max_{B_r} u \geq c(L)r^{1/2}, \quad \text{for all $r\leq 1.$}$$
\end{lem}
\begin{proof}
Since $$u(e_n) \leq C dist(e_n, F(u))^{1/2} \leq C$$ we can apply Harnack inequality and obtain 
$$u(e_{n+1}) \leq C(L)$$ which gives the first claim (in view of Lemma \ref{optimal}).

By scaling, it suffices to prove the second statement for $r=1.$ 

Let $\mu$ be small depending on $L$ and $X_0 \in \{u=0\} \cap B_{1/2}$ be such that $$\mathcal B_{\mu}(X_0) \subset \{u=0\}$$ and it is tangent to $F(u)$ at $Y_0.$ Let $w$ be the harmonic function in $B_{2\mu}(X_0) \setminus \mathcal B_{\mu}(X_0)$ which is zero on $\mathcal B_{\mu}(X_0)$ and equals 1 on $\p B_{2\mu}(X_0)$. Then, by the maximum principle
$$w\max_{B_1} u \geq u \quad \text{on $B_{2\mu}(X_0)$}.$$ Hence, since $Y_0$ is a regular point for $F(u)$ we obtain from the free boundary condition at $Y_0$
$$\max_{B_1}u \;\frac{\p w}{\p U_0}(Y_0) \geq 1\Rightarrow \max_{B_1} u \geq c(\mu).$$
\end{proof}

In view of Proposition \ref{compactness} and the previous lemma we obtain the following compactness result for viscosity solutions with Lipschitz free boundaries.

\begin{cor}\label{cor62}Let $u_k$ be a sequence of viscosity solutions in $B_2$ with $F(u_k)$ uniformly Lipschitz, $0 \in F(u_k)$. Then there exists a subsequence $u_{k_l}$ such that $$u_{k_l} \to u_*, \quad F(u_{k_l}) \to F(u_*) \quad \text{uniformly in $B_1$}$$ with $u_*$ a viscosity solution in $B_1$. \end{cor}

Next we show that positive harmonic functions $v$ (not necessarily viscosity solutions) are monotone in the $e_n$ direction in a neighborhood of $F(v)$, if $F(v)$ is a Lipschitz graph.

\begin{prop}[Monotonicity around $F(v)$]\label{mon_prop}Assume that $v\geq 0$ solves $\Delta v=0$ in $B^+_1(v)$, and that $F(v)$ is a Lipschitz graph in the $e_n$ direction in $\mathcal B_1$ with Lipschitz constant $L,$ and $0 \in F(v).$ Then $v$ is monotone in the $e_n$ direction in $B_\delta$, with $\delta$ depending on $L$ and $n$. \end{prop}
\begin{proof} Assume by scaling that $v$ is defined in $B_{8L}$. Let $w$ be the harmonic function in $$\Omega := \{|(x', 0, x_{n+1})| \leq 1, |x_n| \leq 2L\} \setminus \{v=0\},$$ such that $$w=0 \quad \text{on $\p \Omega \setminus \{x_n=2L\} $}, \quad w=1 \quad \text{on $\{x_n =2L\} \cap  \p \Omega$ }. $$ Then $w$ is strictly increasing in the $e_n$ direction in $\Omega$ (by the maximum principle $w(X) \leq w(X+\eps e_n)$). By boundary Harnack inequality (\cite{CFMS}) $$\frac {v}{w} \in C^{\alpha}(B_{1/2}).$$ After multiplying $v$ by an appropriate constant we may assume that $\dfrac{v}{w}(0) =1$, and
obtain
$$|\frac  v w -1 |\leq \eps \quad \text{in $B_{2\delta}$},$$ for some $\eps$ small to be made precise later and $\delta$ depending on $\eps$, $L$ and $n$. 
For each $r \leq  \delta,$ let 

$$\tilde v(X) = \frac{v(rX)}{w(re_n)}, \quad \tilde w (X) =\frac{w(rX)}{w(re_n)}.$$

Hence $$|\frac{\tilde v}{ \tilde w} -1| \leq  \eps \quad \text{in $B_2$}, \quad \tilde w(e_n)=1.$$

In the region $$\mathcal C_{\mu_0}:= \{|x'| < \mu_0, 1-\mu_0 < |(x_n,x_{n+1})| < 1+\mu_0\} \setminus \{(x,0) \ 
| \  x_n <0\}$$ with $\mu_0$ small depending on $L$, we have (by Harnack inequality for $\tilde w$)$$|\tilde v-\tilde w| \leq \eps \tilde w \leq C(L)\eps.$$ Since $\tilde v-\tilde w$ is harmonic we obtain $$|\tilde v_n -\tilde w_n| \leq C(L)\eps \quad \text{in $\mathcal C_{\frac 3 4 \mu_0}.$}$$ Using that $\tilde v_n -\tilde w_n$ and $\tilde w_n$ are harmonic functions which vanish on $$
\p \mathcal C_{\mu_0} \cap \{x_n \leq 0,x_{n+1}=0\}$$ and $\tilde w_n \geq 0$ and $\tilde w_n(e_n) \geq c
(L) >0$ we obtain that \be\label{uw}|\tilde v_n - \tilde w_n| \leq C(L)\eps \tilde w_n \quad \text{in $\mathcal C_{\frac{\mu_0}{2}}$}.\ee   The bound $\tilde w_n(e_n) \geq c(L) >0$ follows from Harnack inequality for $\tilde w _n$. Indeed, $\tilde w(e_n)=1$ and $\tilde w (-e_n)=0$ thus we can find a point $\bar X$ on the line segment $$[-e_n + \eta e_{n+1}, e_n+\eta e_{n+1}], \quad \text{$\eta$ small }$$ where $\tilde w_n(\bar X) \geq c>0$ for some $c, \eta$ depending on $L.$

From \eqref{uw} we get $$\tilde v_n \geq \tilde w_n(1-C(L)\eps)>0 \quad \text{in $\mathcal C_{\frac{\mu_0}{2}}$},$$ provided that $\eps$ is chosen small depending on $L$. This inequality applied for all $r \leq \delta$ easily implies the conclusion.
\end{proof}

The key step in the proof of Theorem \ref{thm2} is to show that there are no non-trivial Lipschitz  viscosity cones. By the dimension reduction argument in the previous section, it suffices to prove that there are no non-trivial cones with $C^{2,\alpha}$ free boundary outside of the origin. Indeed we remark that Proposition \ref{prop} also holds for viscosity solutions, which can be easily checked directly from Definition \ref{defvisc}. Therefore, Lemma \ref{dimred} holds also for Lipschitz viscosity cones (see Remark \ref{rem2}).

\begin{prop}\label{trivialcones} All Lipschitz viscosity cones are trivial.
\end{prop}
\begin{proof} Let $U$ be a viscosity cone with Lipschitz free boundary and denote by $L$ the Lipschitz norm of $F(U),$ as a graph in the $e_n$ direction. We want to show that $U$ is trivial. By the discussion above we can assume that $F(U)$ is $C^{2,\alpha}$ outside of the origin. 

Now we prove the proposition by induction on $n$. The case $n=1$ is obvious. Assume the statement holds for $n-1.$

By Proposition \ref{mon_prop}, $U$ is monotone in the cone of directions $(\xi,0) \in \mathcal C \times \{0\}$ with $$\mathcal C:= \{\xi=(\xi', \xi_n) \in \R^n  \ | \ \xi_n \geq L |\xi'|\},$$ since $F(U)$ is a Lipschitz graph with respect to any direction $\xi \in \mathcal C^o.$  Moreover there is a direction $\tau \in \p \mathcal C, |\tau| =1$ such that $\tau$ is tangent to $F(U)$ at some point $X_0 \in F(U) \setminus \{0\}.$ Then, $$U_\tau \geq 0 \quad \text{in $\{U>0\}$}.$$

If $U_\tau=0$ at some point in $\{U>0\}$ then $U_\tau \equiv 0$, thus $U$ is constant in the $\tau$ direction, and by dimension reduction we can reduce the problem to $n-1$ dimensions thus by the induction assumption $U$ is trivial. Otherwise $U_\tau >0$ in $\{U>0\}$ and by boundary Harnack inequality  $$U_\tau \geq \delta U \quad \text{in a neighborhood of $X_0$, for some $\delta >0.$}$$ This contradicts Lemma \ref{utau} since for all $r$ small $$\frac \delta 2 r^{1/2} \leq \delta U(X_0 + \nu r) \leq U_\tau(X_0 + \nu r) \leq K r^{1/2 + \alpha.}$$ 
\end{proof}

\begin{rem}
As mentioned in the introduction, the argument above works also for the classical one-phase problem and the minimal surface equation. In the classical one-phase problem we need to use Hopf lemma and in the minimal surface equation we use the strong maximum principle.
\end{rem}
We are now finally ready to exhibit the proof of our main Theorem \ref{thm2}.

\

\textit{Proof of Theorem $\ref{thm2}$.}  First, we show that given a viscosity solution $u$ with Lipschitz free boundary in $B_1$, $0 \in F(u)$, we can find $\sigma>0$ small depending on $u$ such that $F(u)$ is a $C^{2,\alpha}$ graph in $B_{\sigma}.$ Indeed, there exists a blow-up sequence $u_{\lambda_k}$ which converges to a Lipschitz viscosity cone (see Remark \ref{rem2}), that in view of the previous lemma is trivial. The conclusion now follows from our flatness Theorem \ref{mainT} and Corollary \ref{cor62}.

Next we use compactness to show that $\sigma$ depends only on the Lipschitz constant $L$ of $F(u).$ For this we need to show that $F(u)$ is $\bar \eps$-flat in $B_r$ for some $r \geq \sigma$ depending on $L$. If by contradiction no such $\sigma$ exists then we can find a sequence of solutions $u_k$ and of $\sigma_k \to 0$ such that $u_k$ is not $\bar \eps$-flat in any $B_r$ with $r \geq \sigma_k.$ Then the $u_k$ converge uniformly (up to a subsequence) to a solution $u_*$ and we reach a contradiction since $F(u_*)$ is $C^{2,\alpha}$ in a neighborhood of 0 by the first part of the proof.
 \qed

\end{document}